\documentclass{amsart}
\title[Conditional plasticity]{Conditional plasticity of the unit ball of the \(\ell_\infty\)-sum of finitely many strictly convex Banach~spaces}
\author{Kaarel August Kurik}
\address{Institute of Mathematics and Statistics, University of Tartu, Narva
mnt 18, 51009 Tartu, Estonia}
\email{kaarelaugustkurik@gmail.com}

\date{}

\usepackage[colorlinks=true,linkcolor=blue,urlcolor=blue,citecolor=blue,filecolor=black]{hyperref}
\usepackage[citestyle=numeric, backend=biber, maxbibnames=99, bibstyle=numeric]{biblatex}
\usepackage{amssymb, mathtools}
\usepackage{physics}
\newtheorem{theorem}{Theorem}[section]
\newtheorem{lemma}{Lemma}[section]

\theoremstyle{definition}

\newtheorem{definition}{Definition}[section]

\newtheorem*{prob*}{Problem}

\NewDocumentCommand{\ihom}{}{g}
\NewDocumentCommand{\lip}{}{f}
\NewDocumentCommand{\RR}{}{\mathbb{R}}
\NewDocumentCommand{\NN}{}{\mathbb{N}}
\NewDocumentCommand{\ZZ}{}{\mathbb{Z}}
\NewDocumentCommand{\QQ}{}{\mathbb{Q}}
\NewDocumentCommand{\pih}{}{\hat{\pi}}
\NewDocumentCommand{\annotation}{}{(Geometric)}
\NewDocumentCommand{\card}{m}{\abs{#1}}
\NewDocumentCommand{\bcn}{}{B_*}
\NewDocumentCommand{\bbox}{}{B_\square}
\NewDocumentCommand{\xfam}{}{(x_i)_{i \in [2^n]}}
\NewDocumentCommand{\yfam}{}{(y_i)_{i \in [2^n]}}
\newcommand{\cupdot}{\mathbin{\mathaccent\cdot\cup}}

\newcommand{\clo}[1]{\overline{#1}}

\begin{document}

\begin{abstract}
    We prove that for any $\ell_\infty$-sum $Z = \bigoplus_{i \in [n]} X_i$ of finitely many strictly
    convex Banach spaces $(X_i)_{i \in [n]}$, an extremeness preserving 1-Lipschitz bijection
    $\lip\colon B_Z \to B_Z$ is an isometry, by constraining the componentwise behavior of the inverse
    $\ihom=\lip^{-1}$ with a theorem admitting a graph-theoretic interpretation.
    We also show that if $X, Y$ are Banach spaces, then a bijective 1-Lipschitz non-isometry of
    type $B_X \to B_Y$ can be used to construct a bijective 1-Lipschitz non-isometry of type
    $B_{X'} \to B_{X'}$ for some Banach space $X'$, and that a homeomorphic 1-Lipschitz non-isometry of
    type $B_X \to B_X$ restricts to a homeomorphic 1-Lipschitz non-isometry of type
    $B_S \to B_S$ for some separable subspace $S \leq X$.
\end{abstract}

\maketitle

\section{Introduction}

The central aim of this article is to present a generalization of a key lemma
found in Nikita Leo's proof of the plasticity of the closed unit ball of
the $\ell_\infty$-sum of two strictly convex Banach spaces, where the generalization
extends the lemma to the $\ell_\infty$-sum of any finite number of strictly convex Banach
spaces by way of a graph-theoretic analogue. In addition, the generalized lemma is applied to prove that any 1-Lipschitz bijection from
the closed unit ball of such a space to itself
which maps extreme points to extreme points
or the sphere into itself must be an isometry.

Two additional results are also presented, which may prove no less important for the study of plasticity than the main result. The first states that the existence of a non-isometric 1-Lipschitz bijection between the unit balls of two distinct Banach spaces implies the existence of a non-isometric 1-Lipschitz bijection from the unit ball of some Banach space to itself, thereby proving that the general question of unit ball plasticity for Banach space pairs is equivalent to unit ball plasticity for single Banach spaces. The second states that the homeomorphic plasticity of the unit ball for a Banach space is equivalent to the homeomorphic plasticity of all of the space's separable subspaces.

\section{Preliminaries and notation}

\subsection{Background}

The notion of plasticity for metric spaces was introduced by Naimpally, Piotrowski, and Wingler in their 2006 article \autocite{naimpally:2006}.
A metric space is said to be \textit{EC-plastic} (or just \textit{plastic}) when
all 1-Lipschitz bijections from the space into itself are isometries.

In their 2016 article \autocite{cascales:2016}, Cascales, Kadets, Orihuela, and Wingler began an investigation
of the following question.

\begin{prob*}
    Is the closed unit ball of every Banach space plastic?
\end{prob*}

In said article, this question was answered affirmatively for the special case of strictly convex Banach spaces. (Recall that a Banach
space is strictly convex if its
unit sphere contains no segments with distinct
endpoints.) The general case, however, remains open.

All totally bounded metric spaces are known to be plastic, including the unit balls of finite-dimensional Banach spaces \autocite{naimpally:2006}.

The unit ball is also known to be plastic
in the following cases:
\begin{itemize}
    \item spaces whose unit sphere is a union of finite-dimensional polyhedral extreme subsets (incl. all strictly convex Banach spaces) \autocite{angosto:2019,cascales:2016},
    \item any $\ell_1$-sum of strictly convex Banach spaces 
    (incl. $\ell_1$ itself) \autocite{kadets_zavarzina:2016,kadets_zavarzina:2018},
    \item the $\ell_\infty$-sum of *two* strictly convex Banach spaces \autocite{haller:2022},
    \item $\ell_1 \oplus \RR$ \autocite{haller:2022},
    \item $C(K)$, where $K$ is a compact metrizable space with finitely many
    accumulation points (incl. $c \cong C(\omega+1)$, i.e. the space of convergent real sequences)
    \autocite{fakhoury:2024,leo:2022}.
\end{itemize}

In \autocite{haller:2022}, it was shown (with proof essentially due to Nikita Leo)
that the $\ell_\infty$-sum of two strictly
convex Banach spaces has a plastic unit ball. While the proof does not directly apply to an arbitrary finite sum of strictly convex Banach spaces, a crucial step in the proof can be modified to suit this purpose. By generalizing this step, we can establish a similar but weaker property than plasticity, which only considers a specific class of
1-Lipschitz bijections that is well-behaved with respect to extreme points.

\subsection{Conventions, notation}

We adopt the conventions that $0 \in \NN$ and $[n] = {i \in \NN \colon i < n}$.

For any map $f$ from a metric space $(M,d)$ to itself,
we say that it is \textit{non-expansive} when it is a 1‑Lipschitz
map, and that it is \textit{non-contractive} when for all
$x, y \in M$, we have $d(x,y) \leq d(f(x), f(y))$ (note
that this is dual to the inequality $d(x,y) \geq d(f(x), f(y))$ defining
1‑Lipschitz maps).

\section{Standalone results}

We begin with two theorems relating to plasticity which can be stated and proved without much preamble, and which are independent of both each other and the remainder of the article.

\NewDocumentCommand{\induced}{}{\lip'}

\begin{theorem}
    Suppose there are Banach spaces $X, Y$, and a non-expansive bijection $\lip \colon B_X \to B_Y$ such that $\lip$ is not an isometry. Then there is a Banach space $Z$ and a non-expansive bijection $\induced \colon B_Z \to B_Z$ such that $\induced$ is not an isometry.
\end{theorem}

This result is motivated by the work of Olesia Zavarzina in \autocite{zavarzina:2017}.

\begin{proof}
    Let $C_i$ be a Banach space for each $i \in \ZZ$, such that $C_i = X$
    for $i < 0$ and $C_i = Y$ for $i \geq 0$. Take $Z \coloneqq \bigoplus_{i=-\infty}^\infty C_i$ with the $\infty$-norm. Define $\induced \colon B_Z \to B_Z$ by $\pi_i \induced(z) = \pi_{i-1} z$ for $i \neq 0$ and $\pi_0 \induced(z) = \lip(\pi_{-1}z)$.
  
    It is clear by inspection that the codomain of $\induced$ is correct and that it is a non-expansive bijection. That it is not an isometry follows from considering the natural inclusions of two points $x, x' \in C_{-1}$ into $Z$, where $\norm{\induced(x) - \induced(x')} = \norm{\lip(x)-\lip(x')} < \norm{x-x'}$.
\end{proof}

\begin{lemma}
    \label{rat-scaling}
    Let $X$ be a Banach space and let $A \subseteq X$ be closed under scaling by rationals. Then $\clo{A \cap B_X} = \clo{A} \cap B_X$.
\end{lemma}

\begin{proof}
Since $A \cap B_X \subseteq \clo{A} \cap B_X$ and the latter is closed, we
have $\clo{A \cap B_X} \subseteq \clo{A} \cap B_X$. It thus suffices to show
the opposite inclusion. Fix any $a \in \clo{A} \cap B_X$ and a sequence $a_i
\in A$ that converges to $a$. If $\norm{a} < 1$, then $a_i \in A \cap B_X$ for
all sufficiently large $i$, from which $a \in \clo{A \cap B_X}$. If $\norm{a} =
1$, then choose a sequence of rationals $q_i \in \QQ$ such that $\abs{q_i} \leq 1/
\norm{a_i}$ for all sufficiently large $i$, and $q_i \to 1$. This is possible
since $\norm{a_i} \to \norm{a} = 1$. We then have $q_i a_i \in A \cap B_X$ for all
sufficiently large $i$, and $q_i a_i \to a$, so $a \in \clo{A \cap B_X}$. We thus
have that $\clo{A} \cap B_X \subseteq \clo{A \cap B_X}$, so $\clo{A} \cap B_X =
\clo{A \cap B_X}$.
\end{proof}

\newcommand\restr[2]{\ensuremath{\left.#1\right|_{#2}}}

\begin{theorem}
    Let $X$ be a Banach space and $\lip \colon  B_X \to B_X$ be a non-expansive homeomorphism that is not an isometry. Then $X$ has a separable closed subspace $Y$ such that $\lip(B_Y) = B_Y$ and $\rho \coloneqq \restr{\lip}{B_Y}  \colon  B_Y \to B_Y$ is a non-expansive homeomorphism that is not an isometry.
\end{theorem}

\begin{proof}
    Let $x, x' \in B_X$ be points for which $\norm{\lip(x)-\lip(x')} < \norm{x-x'}$.

    Define the set function $H  \colon  2^X \to 2^X$ as \[ H(S) = \lip(S \cap B_X) \cup \lip^{-1}(S \cap B_X) \cup \QQ \cdot S \cup (S+S). \] Note that $H(S)$ is countable whenever $S$ is countable, $S \subseteq H(S)$. Moreover, since $H$ is a union of set functions which are monotonic and continuous with respect to ascending chains of set inclusions, then $H$ is monotonic and continuous with respect to ascending chains also.
  
    Define $S_0 = \{x,x'\}$ and $S_{n+1} = H(S_n)$ for $n \in \NN$. Let $L = \bigcup_{n=0}^\infty S_n.$ Since $L$ is the limit of an ascending chain, we have $H(L) = \bigcup_{n=0}^\infty H(S_n) = \bigcup_{n=0}^\infty S_{n+1} = L$, so $L$ is a fixed point of $H$. Since $L$ is a countable union of countable sets, then $L$ is itself countable.
  
    Since $L$ is a fixed-point of $H$, we have that it is closed under addition and rational scaling, from which $\clo{L}$ is closed under addition and real scaling, so $\clo{L}$ is a closed subspace of $X$.
    Since $L$ is countable, $\clo{L}$ is separable. By \autoref{rat-scaling},
    we have that $\clo{L \cap B_X} = \clo{L} \cap B_X = B_{\clo{L}}$.
  
    Since $\lip$ is continuous and $L$ is closed under $\lip$, we have
    that $\lip(\clo{L \cap B_X}) \subseteq \clo{\lip(L \cap B_X)} \subseteq \clo{L \cap B_X}$, so $\lip(B_{\clo{L}}) \subseteq B_{\clo{L}}$. Analogously, we have $\lip^{-1}(B_{\clo{L}}) \subseteq B_{\clo{L}}$. From these, we have $\lip(B_{\clo{L}}) = B_{\clo{L}}$, so $\rho$ is a well-defined non-expansive homeomorphism. Since $x, x' \in B_{\clo{L}}$, we also have that $\rho$ is not an isometry.
\end{proof}

\section{Primary results} \label{sec:main}

\subsection{Conventions, notation} \label{sec:main_notation}

Throughout \autoref{sec:main}, we consider the following structure
and a certain weakening thereof, explained below:

\begin{itemize}
    \item $n$ is a fixed value in $\NN$ such that $n \geq 1$.
    \item $X_i$ is a family of strictly convex nontrivial Banach spaces, indexed by $i \in [n]$.
    \item $B_i, S_i$ are the unit ball and sphere of $X_i$ respectively.
    \item $Z = \bigoplus_{i \in [n]} X_i$ is the direct sum of the family $X_i$ endowed with the $infinity$-norm, i.e. if $z = (x_0, \ldots, x_{n-1}) \in Z$, then $\norm{z} = \max_{i \in [n]} \norm{x_i}$.
    \item $\pi_i \colon Z \to X_i$ is the projection onto the $i$-th component of $Z$.
    \item More generally, $\pi_i$ should be understood as the projection onto the $i$-th component of \textit{any} structure with components indexed by a set containing $i$.
    \item $\pih_i \colon Z \to \hat{X}_i$ is the complementary projection of the $i$-th component, where $\hat{X}_i = \bigoplus_{j \in [n] \setminus {i} }X_j$. For all $j \in [n] \setminus i$, we have that $\pi_j z = \pi_j \pih_i z$, and $\pi_i \pih_i z$ is ill-defined.
    \item $B_Z, S_Z$ are the unit ball and sphere of $Z$ respectively.
    \item $E \subseteq B_Z$ is the set of extreme points in $B_Z$. This can be shown to be those $z \in S_Z$ for which $\forall i \in [n], z_i \in S_i$.
    \item $\ihom \colon B_Z \to B_Z$ is a non-contractive injection from $B_Z$ to itself.
    \item $\lip \colon B_Z \to B_Z$ is a 1-Lipschitz bijection, such that $\lip = \ihom^{-1}$ if $\ihom$ is invertible.
    \item $\sigma \colon [n] \to [n]$ is a permutation of $[n]$.
    \item $g_i \colon S_i \to S_{\sigma(i)}$ is a non-contractive injection satisfying $\pi_i x \in S_i \Rightarrow \ihom(x)_{\sigma(i)} = \ihom_i (\pi_i x)$. Proving that such functions exist is one of the central results of this article.
    \item $\lip_i \colon S_{\sigma(i)} \to S_i$ is the inverse of $\ihom_i$, whenever the inverse exists.
\end{itemize}

Many of our results require no geometric considerations, and thus can be carried out on an analogous graph structure:

\begin{itemize}
    \item $B_i$ is a disconnected graph with maximal vertex degree exactly 1. This means every connected component of $B_i$ consists of either a single vertex or two vertices joined by an edge.
    \item $S_i \leq B_i$ is the subgraph of $B_i$ consisting only of vertices with a neighbor.
    \item $B_*$ is the co-normal product of the graphs $B_i$, while $B_\square$ is the Cartesian product of the same. $B_\square$ is a subgraph of $B_*$. Their common set of vertices is denoted $B$.
    \item $E$ is the set of those vertices $e \in B$ for which $\forall i \in [n], e_i \in S_i$. $E_* \leq B_*$ and $E_\square \leq B_\square$ are the respective induced subgraphs.
    \item $S$ is the set of those vertices $s \in B$ for which $\exists i \in [n], s_i \in S_i$. $S_* \leq B_*$ and $S_\square \leq B_\square$ are the respective induced subgraphs.
    \item The adjacency of vertices $u,v \in B_*$ is denoted $u \sim v$, and adjacency in $B_\square$ is denoted $u \sim' v$.
    \item The adjacency of vertices $u,v \in B_i$ is denoted either $u \sim v$ or $u \sim' v$, since the two product structures agree on $B_i$.
    \item $\ihom  \colon  B_* \to B_*$ is an injective graph homomorphism.
    \item $\lip  \colon  B_* \to B_*$ is the set-theoretic inverse of $\ihom$ whenever $\ihom$ is invertible.
    \item $\ihom_i  \colon  S_i \to S_{\sigma(i)}$ is a local isomorphism satisfying $\pi_i x \in S_i \Rightarrow \ihom(x)_{\sigma(i)} = \ihom_i (\pi_i x)$.
    \item $\lip_i  \colon  S_{\sigma(i)} \to S_i$ is the set-theoretic inverse of $\ihom_i$, whenever this exists.
\end{itemize}

The analogies in notation between the Banach space structure and the graph structure are motivated by these identifications:

\begin{itemize}
    \item Given points $a,b \in B_i$ with $B_i$ the unit ball of $X_i$, we may construct the graph $B_i$ by setting $a \sim b$ iff $\norm{a-b} = 2$. This will be a disconnected graph with maximal vertex degree exactly 1.
    \item $B_*$ is the graph constructed from $B_Z$ with precisely the same condition for adjacency, while $B_square$ has the adjacency condition $u \sim v \Leftrightarrow \exists i \in [n], (\norm{u_i - v_i} = 2) \land (\pih_i u = \pih_i v)$.
    \item Each subgraph $S_i$ of $B_i$ is generated by the unit sphere of $B_i$. Analogously, $S,E$ are the vertex sets corresponding to the unit sphere $S_Z$, and the extreme points of $B_Z$, respectively.
\end{itemize}

The remaining analogies are left for the reader to verify.

Throughout \autoref{sec:main}, every theorem and lemma shall have an annotation indicating whether it requires the Banach space structure. If the graph-theoretic structure is sufficient, there will be no annotation.

\subsection{Graph-theoretic results}

We would like to prove the following theorem:

\begin{theorem} \label{thm:factors}
Let $\ihom  \colon  \bcn \to \bcn$ be an injective homomorphism. Then there exists a permutation $\sigma \colon  [n]\to[n]$ and a family of local isomorphisms $g_i  \colon  S_i \to S_{\sigma(i)}$ such that for all $x \in \bcn$ and all $i \in [n]$, we have $x_i \in S_i \Rightarrow \pi_{\sigma(i)}\ihom(x) = \ihom_i (\pi_i x)$.\end{theorem}

\begin{lemma} \label{lem:clique-ext}

  A clique of $2^n-1$ points in $\bcn$ has at most one extension
to a clique of $2^n$ points.
\end{lemma}

\begin{proof} 

  Induction on $n$. The case with $n=1$ is clear by inspection.

  First note that if the statement is true for a graph $\bcn$, then the maximal clique size for $\bcn$ is at most $2^n$. If there were a clique of $2^n+1$ points, then every subclique of size $2^n - 1$ would have at least two distinct extensions to a clique of size $2^n$.

  Consider an enumerated family $(x_i)_{i \in [2^n]}$ of pairwise adjacent
  vertices in $\bcn$. We want to show that any vertex $q \in \bcn$ that is
  adjacent to all vertices $x_i$ with $i > 0$ is equal to $x_0$.

  Partition the family into $C \cupdot D$ such that $C$ is any maximal
  subset of the family such that $\pi_0 C$ is an edgeless graph.

  Note two things: $\pi_0 D$ is also edgeless, and $\card{C} = \card{D} = 2^{n-1}$. From this, it follows that $D$ also satisfies the defining condition of $C$.

  Let's verify the first observation. If any element $\pi_0 x$ of $\pi_0 D$ were disconnected from $\pi_0 C$, then $C$ could be extended to $C \cup \{x\}$, so $C$ would not be maximal. This implies that every element of $\pi_0 D$ is connected to some element of $\pi_0 C$, so $\pi_0 D \subseteq N(\pi_0 C)$. Since $\pi_0 C$ is edgeless, then $N(\pi_0 C)$ is edgeless,
  thus $\pi_0 D$ is edgeless.

  Now the second observation. Since $C$ is maximal among those subsets $S$ for which $\pi_0 S$ is edgeless, and $\pi_0 D$ is edgeless, we have $\card{C} \geq \card{D}$, from which $\card{C} \geq 2^{n-1}$. Now, since $\forall c, c' \in C, c \sim c'$ while $\pi_0 c \nsim \pi_0 c'$, we must have $\pih_0 c \sim \pih_0 c'$. This means that $\pih_0 C$ is a clique of at least $\card{C} \geq 2^{n-1}$ points in $\pih_0 \bcn$ (we have that $\card{\pih_0 C} = \card{C}$ since $c \neq c' \Rightarrow c \sim c' \Rightarrow \pih_0 c \sim \pih_0 c' \Rightarrow \pih_0 c \neq \pih_0 c'$). By the induction assumption, the largest clique size in $\pih_0 \bcn$ is at most $2^{n-1}$, so $\card{C} \geq 2^{n-1} \geq \card{\pih_0 C} = \card{C}$, hence $\card{C} = 2^{n-1}$.

  Having shown that $D$ is also maximal w.r.t. C's defining condition, we have that $\pi_0 C \subseteq N(\pi_0 D)$. Since $N^2(S) \subseteq S$ for a subset $S$ of a graph of maximal degree 1, we have that $N(\pi_0 C) \subseteq \pi_0 D$, from which $\pi_0 D = N(\pi_0 C)$ and $\pi_0 C = N(\pi_0 D)$.

  Assume WLOG that $x_0 \in C$. If $\pi_0 q$ had no edge to $\pi_0 D$, then
  $\pih_0 D \cupdot \{\pih_0 q\}$ would form a clique of $2^{n-1}+1$ elements
  in $\pih_0 \bcn$, which is impossible. Thus $\pi_0 q \sim \pi_0 b$ for some $b \in D$. It follows that $\pi_0 q$ has no edge to $\pi_0 C$, from which $\pih_0 q$ has edges to all members of $\pih_0 (C - \{x_0\})$. This means $\pih_0 (C - \{x_0\})$ is a clique of $2^{n-1}-1$ points in $\pih_0 \bcn$ which can be extended by $\pih_0 q$ or by $\pih_0 x_0$. By the induction assumption, this means $\pih_0 q = \pih_0 x_0$.
  
  Running the same argument by some index $j > 0$, we also have that $\pih_j q = \pih_j x_0$. Since these two projections cover all components, we must have $q = x_0$.
\end{proof}

For any $x \in E$, define $T(x)$ as the connected component of $x$ in $\bbox$. It may be readily verified that $T(x)$ has $2^n$ members: there is a bijection between subsets $J \subseteq [n]$ and vertices $x_J \in T(x)$, such that $\forall j \in J, \pi_j x_J = \pi_j x$ and $\forall j \in [n]\setminus J, \pi_j x_J \sim \pi_j x$. $T(x)$ is thus naturally isomorphic to a pointed $n$‑hypercube graph.

From here on, let $x_0 \in E$. Let $\xfam = T(x_0)$ and define $y_i = \ihom(x_i)$.

\begin{lemma} \label{one-component}

  $\ihom$ is an $E_\square$-homomorphism.
  \end{lemma}

\begin{proof} 

  Let us have $x_a, x_b \in E$ such that $x_a \sim' x_b$. There exists some $x_0 \in E$ for which $x_a, x_b \in T(x_0)$ --- we may choose $x_0 \coloneqq x_a$.

  Since $\ihom$ is a $\bcn$-homomorphism and $x_a \sim x_b$, we have $y_a \sim y_b$, so $y_a, y_b$ are adjacent in at least one component. Note also that $\yfam$ is a clique in $\bcn$.

  We first show that $y_a, y_b$ differ in exactly one component.
  Suppose for the sake of contradiction that $y_a, y_b$ differ in at least two components, and let the first two of these have indices $i, j$.

  Let $C_i \cupdot D_i$ be a partition of $\yfam$ such that $\pi_i C_i$ is maximal and edgeless, and $y_a \in C_i, y_b \in D_i$. To prove that such a partition exists, it is sufficient to find any $y_c$ such that $\pi_i y_c \sim \pi_i y_b$, start with $\{y_a, y_c\} \subseteq C_i$ and extend $C_i$ to maximality (noting that $\pi_i y_a \neq \pi_i y_b$ implies $\pi_i y_a \nsim \pi_i y_c$). Such a $y_c$ exists, since any partition of $\yfam$ into two maximal $\pi_i$‑edgeless sets consists of two nonempty sets, one of which contains $y_b$, and any member of the other component can be $y_c$.
  Also construct an analogous partition $C_j \cupdot D_j$ for $\pi_j$.

  Let $k$ be the index for which $\pi_k x_a \sim \pi_k x_b$. By the definition of $T$, this is the unique index at which $\pi_k x_a \neq \pi_k x_b$, so $\pih_k x_a = \pih_k x_b$. Fix $v \in B$ such
  that $\pih_k v = \pih_k x_a, \pih_k x_b$, and $\pi_k v \neq \pi_k x_a, \pi_k x_b$. Note that $\ihom(v)$ has $\bcn$-edges to all of $\yfam$ except possibly for $y_a, y_b$.

  Suppose that $\pi_i \ihom(v)$ has an edge to $\pi_i C_i$. This implies $\pi_i \ihom(v) \in N(\pi_i C_i) = \pi_i D_i$,
  from which $\pih_i \ihom(v)$ forms a $\bcn$-clique with $\pih_i (D_i - \{y_b\})$. By \autoref{lem:clique-ext} this implies that $\pih_i \ihom(v) = \pih_i y_b$, from which $\pi_j \ihom(v) = \pi_j y_b$. From this, we have that $\pih_j \ihom(v)$ forms a $\bcn$-clique with $\pih_j (D_j - \{y_b\})$, which implies by \autoref{lem:clique-ext} that $\pih_j \ihom(v) = \pih_j y_b$. Since $\pih_i \ihom(v) = \pih_i y_b$ and $\pih_j \ihom(v) = \pih_j y_b$, we have $\ihom(v) = y_b = \ihom(x_b)$, from which $v = x_b$, which contradicts our choice of $v$.

  Consequently, $\pi_i \ihom(v)$ has no edge to $\pi_i C_i$, and by symmetrical argument it has no edge to $\pi_i D_i$ either. From these, we have that $\pih_i \ihom(v)$ forms a $\bcn$-clique with $\pih_i (C_i - \{y_a\})$ and with $\pih_i (D_i - \{y_b\})$, which implies by \autoref{lem:clique-ext} that $\pih_i y_a = \pih_i \ihom(v) = \pih_i y_b$, so $\pi_j y_a = \pi_j y_b$, contradicting our choice of $j$ and thus our claim that $y_a, y_b$ differ in at least two components.

  Since $y_a$ and $y_b$ are adjacent in at least one component and differ in at most one component, these bounds must be saturated and exactly one component accounts for both, which is some $m$ at which $\pi_{m} y_a \sim \pi_{m} y_b$ and $\pih_{m} y_a = \pih_{m} y_b$ --- consequently $y_a \sim' y_b$.
\end{proof}

\begin{lemma} 

  There is a permutation $\sigma \colon  [n]\to[n]$ and a family of local isomorphisms $\ihom_i  \colon  S_i \to S_{\sigma(i)}$ such that for $e \in E$, we have $\pi_{\sigma(i)} \ihom(e) = \ihom_i (\pi_i e)$.
\end{lemma}

\begin{proof} 

First, note that $\yfam$ is a subgraph of $T(y_0)$, since $\yfam$ is connected (on account of being a homomorphic image of $T(x_0)$), and $T(y_0)$ is the connected component of $y_0$ in $\bbox$. Consequently, $\ihom$ is a homomorphism from $T(x_0)$ to $T(y_0)$.

Because $\ihom$ is an injective homomorphism between the finite graphs $T(x_0)$
and $T(y_0)$, and the two graphs have an equal number of edges, $\ihom$ must be an isomorphism from $T(x_0)$ to $T(y_0)$.

That this isomorphism arises from a componentwise isomorphism (up to permutation of components) follows from \autocite[Theorem 6.8]{handbook}.

Since these isomorphisms must glue compatibly across all of $E$, we have that there exists a permutation $\sigma  \colon  [n] \to [n]$ and a family of local isomorphisms $\ihom_i  \colon  S_i \to S_{\sigma(i)}$ such that for all $x \in E$, $\pi_{\sigma(i)} \ihom(x) = \ihom_i (\pi_i x)$.
\end{proof}

We would like to extend this slightly to the case of vertices outside of $E$ to finish our proof of  \autoref{thm:factors}.

\begin{lemma} \label{lem:interior}

  Let $q \in B$ and $i \in [n]$ be such that $\pi_i q = \pi_i x_0$. Then $\pi_{\sigma(i)} \ihom(q) = \pi_{\sigma(i)} \ihom(x_0) = \ihom_i (\pi_i x_0)$.
  \end{lemma}

\begin{proof} 

  This is trivial for the $n=1$ case (since then $q \in E$ or the claim is vacuous), so we will assume $n > 1$.

  Let $(x_j)_{j \in J}$ be the subfamily of $\xfam$ for that satisfies $\forall j \in J, \pi_i x_j \sim \pi_i x_0$. We then have that $\card{J} = 2^{n-1}$ and $\forall j \in J, \pi_{\sigma(i)}y_j \sim \pi_{\sigma(i)} y_0$. Note also that since $q \sim x_j$, we have $\ihom(q) \sim y_j$ for each $j \in J$.

  Suppose that $\pi_{\sigma(i)} \ihom(q) \neq \pi_{\sigma(i)} y_0$. This implies that $\pi_{\sigma(i)}\ihom(q) \nsim \pi_{\sigma(i)}y_j$ for each $j \in J$, from which (by way of $\ihom(q) \sim y_j$) we must have $\pih_{\sigma(i)}\ihom(q) \sim \pih_{\sigma(i)}y_j$. This induces a clique of $2^{n-1}+1$ vertices $\{\pih_{\sigma(i)}\ihom(q)\} \cup \{\pih_{\sigma(i)}y_j  \colon  j \in J\}$ in $\pih_{\sigma(i)}\bcn$, which is impossible by \autoref{lem:clique-ext}. We thus have that $\pi_{\sigma(i)} \ihom(q) = \pi_{\sigma(i)}y_0 = \ihom_i (\pi_i x_0)$.
\end{proof}

Since for each $x_i \in S_i$ there exists an $x \in E$ with $\pi_i x = x_i$, \autoref{lem:interior} concludes our proof of  \autoref{thm:factors}.

\subsection{Applications to plasticity}

We begin with a straightforward corollary of  \autoref{thm:factors}.

\begin{theorem} \label{thm:banach-factors}
\annotation  Let $Z \coloneqq \bigoplus_{i \in [n]} X_i$ and let $\ihom \colon B_Z \to B_Z$ be a non-contractive function. Then there is some permutation $\sigma \colon [n] \to [n]$ and a family of non-contractive functions $\ihom_i \colon S_i \to S_{\sigma(i)}$ such that for all points $x \in B_Z$ and all $i \in [n]$ we have 
  $\pi_i x \in S_i \Rightarrow \pi_{\sigma(i)} G(x) = \ihom_i (\pi_i x)$.
\end{theorem}

\autoref{thm:banach-factors} follows from applying  \autoref{thm:factors} to the graph with vertices in $B_Z$, edges $\{\{x,x'\} \colon \norm{x-x'}=2\}$ and $\ihom$ as the injective homomorphism.

This can be applied to prove some more natural theorems concerning plasticity.

\begin{theorem} \label{thm:natural}
\annotation
  Let $\lip\colon B_Z \to B_Z$ be a 1-Lipschitz bijection. If $\lip$ maps extreme points to extreme points, or $\lip(S_Z) \subseteq S_Z$, then $\lip$ is an isometry.
\end{theorem}

Our proof of \autoref{thm:natural} draws upon Nikita Leo's work in \autocite{haller:2022}
for its outline.

We begin with some graph-theoretic lemmas mirroring the conditions of \autoref{thm:natural}.

\begin{lemma} \label{lem:bijective-factors}

  If $S \subseteq \ihom(S)$ or $E \subseteq \ihom(E)$, then each $\ihom_i$ is a bijection.
\end{lemma}

Note that the conditions $S \subseteq \ihom(S)$ and $E \subseteq \ihom(E)$ in \autoref{lem:bijective-factors} are equivalent to $\lip(S) \subseteq S$ and $\lip(E) \subseteq E$ respectively in the setting of \autoref{thm:natural}.

\begin{proof} 

  It is enough to show that each $\ihom_i$ is surjective, i.e. that $y \in \ihom_i (S_i)$ for any $y \in S_{\sigma(i)}$.

  We shall first consider the case where $S \subseteq \ihom(S)$.

  First, fix a point $q \in S$ such that $\pi_{\sigma(i)} q = y$ and $\pi_j q \notin S_j$ for all $j \neq \sigma(i)$. By $S \subseteq \ihom(S)$, we may fix $x \in S$ such that $\ihom(x) = q$. Let $J = \{j \in [n] \colon \pi_j x \in S_j\}$. From  \autoref{thm:factors}, we know that $\pi_{\sigma(j)} \ihom(x) \in S_{\sigma(j)}$ for all $j \in J$. It follows that $J \subseteq \{i\}$. Since $x \in S$, we have that $J$ is nonempty, from which $J = \{i\}$. This allows us to use  \autoref{thm:factors} to conclude that $\ihom_i (\pi_i x) = \pi_{\sigma(i)} \ihom(x) = y$, so $y \in \ihom_i (S_i)$.

  We now consider the case where $E \subseteq \ihom(E)$.

  Fix a point $q \in E$ such that $\pi_{\sigma(i)} q = y$. Since $q \in E$, there is some $x \in E$ for which $\ihom(x) = q$, and by  \autoref{thm:factors}, we have $\ihom_i (\pi_i x) = \pi_{\sigma(i)} \ihom(x) = y$.
\end{proof}

\begin{lemma} \label{lem:sphere-implies-extreme}

  If $S \subseteq \ihom(S)$, then $E \subseteq \ihom(E)$.
\end{lemma}

\begin{proof} 

  Fix any point $y \in E$. We aim to construct a point $x \in E$ such that $\ihom(x)=y$.

  By $S \subseteq \ihom(S)$ and \autoref{lem:bijective-factors}, we have that each $\ihom_i$ is a bijection, so we may define $x$ such that $\ihom_i (\pi_i x) = \pi_{\sigma(i)} y$ for each $i \in [n]$. By  \autoref{thm:factors}, we have that $\pi_{\sigma(i)} \ihom(x) = \ihom_i (\pi_i x) = \pi_{\sigma(i)} y$, so $\ihom(x) = y$ as desired.
\end{proof}

We now proceed with more geometric results for which a graph-theoretic analogue has not been recovered. Whenever \autoref{lem:bijective-factors} is applicable, we define $\lip_i = \ihom_i^{-1}$ in analogy with the relation $\lip = \ihom^{-1}$.

For convenience in what follows, we define the functions $\gamma_i  \colon B_i \to B_{\sigma(i)}$ such that
$\gamma_i (x) =
\begin{cases}\norm{x} \ihom_i (\frac{x}{\norm{x}}) &\text{if}\ x \neq 0\\ 0 &\text{otherwise}\end{cases}$, and $\phi_i \colon B_{\sigma(i)} \to B_i$ as $\phi_i = \gamma_i^{-1}$. We additionally define
$\gamma, \phi \colon B_Z \to B_Z$ as $\pi_{\sigma(i)} \gamma(x) = \gamma_i (\pi_i x)$, and
$\pi_i \phi(x) = \phi_i (\pi_{\sigma(i)} x)$.

The reader may readily verify that $\phi = \gamma^{-1}$.

\begin{definition} 

  Given a $1$-Lipschitz bijection $\lip \colon B_Z \to B_Z$ and $\ihom = \lip^{-1}$ satisfying the conditions of \autoref{lem:bijective-factors}, $\ihom$ is said to be \textit{homogeneous in $k$ components} if for all $x \in B_Z$ such that $x$ has norm $1$ on at least $n-k$ components, we have $\ihom(x) = \gamma(x)$. Analogously, we say that $\lip$ is homogeneous in $k$ components when, for the same $x$, we have $\lip(x) = \phi(x)$.
\end{definition}

\begin{lemma} \label{lem:homogeneity-equiv}
\annotation
  If $\lip(E) \subseteq E$, then
  the function $\ihom$ is homogeneous in $k$ components if and only if
  $\lip$ is also.
\end{lemma}

\begin{proof} 

  We will show that $\ihom$ being homogeneous in $k$ components implies that $\lip$ is also. The proof of the converse is analogous.

  Let us have $J \subseteq [n]$ with $\card{J} \leq k$ and $x \in B_Z$ such that $x$ has norm 1 on components $[n] \setminus J$.

  First, define $a_i = \norm{\pi_i x}$ and fix $y \in E$ such that
  $\pi_i x = a_i \pi_i y$. Then fix $q \in B_Z$ such that $\pi_i q = a_{\sigma(i)} \pi_i \lip(y)$, and note that $q$ has norm 1 in exactly as many components as $x$, since $f(y) \in E$. Consequently, the homogeneity of $\ihom$ in $k$ components applies to $q$, from which

  \[ \pi_{\sigma(i)} \ihom(q) = \gamma_i (\pi_i q) = 0 = \pi_{\sigma(i)} x "      if"
  \pi_i q = 0 \]
  and
  \[ \pi_{\sigma(i)} g(q) = \gamma_i (\pi_i q) = \norm{\pi_i q} \ihom_i (\frac{\pi_i q}{\norm{\pi_i q}}) = a_{\sigma(i)} \ihom_i (\pi_i \lip(y)) =_4 a_{\sigma(i)} \pi_{\sigma(i)} y = \pi_{\sigma(i)} x \]

  otherwise, where equality 4 follows from $\lip(y) \in E$ and \autoref{thm:banach-factors}.
  Consequently, we have $\ihom(q) = x$ and $\gamma(q) = x$, from which
  $q = f(x)$ and $q = \phi(x)$, so $f(x) = \phi(x)$.
\end{proof}

\begin{lemma} \label{lem:g-homogeneous}

  \annotation
  If $\lip(E) \subseteq E$, then $\ihom$ is homogeneous in $n$ components, i.e. $g = \gamma$.
\end{lemma}
\begin{proof} 

  We proceed by induction. The base case of $n=0$ is covered by \autoref{thm:banach-factors}.

  Suppose $\ihom$ is homogeneous in $k < n$ components. Let $J \subseteq [n]$ with
  $\card{J} = k+1$, and let $x \in B_Z$ have norm 1 on components $[n] \setminus J$. We
  already know that for $i \in [n]\setminus J$, $\pi_{\sigma(i)} \ihom(x) = \ihom_i (\pi_i x)
  = \pi_{\sigma(i)} \gamma (x)$ by
  the construction of $\ihom_i$.

  It thus suffices to show that for $i \in J$, we have $\pi_{\sigma(i)} \ihom(x) =
  \pi_{\sigma(i)} \gamma(x)$.

	Fix any $i \in J$. Define $y$ such that $\pih_{\sigma(i)} y = \pih_{\sigma(i)} \ihom(x)$ and $\norm{\pi_{\sigma(i)} \ihom(x)} \pi_{\sigma(i)} y = \pi_{\sigma(i)} \ihom(x)$. Note that
	$y$ has norm 1 on $k$ components, so we can apply the induction assumption to it.
	We also define $z$ as equal to $y$, except at $\sigma(i)$, where we flip the
	sign of the component.

	Since $\ihom_j (-x) = -\ihom_j (x)$ holds for all $j \in [n], x \in S_j$, we also have that
	$\lip_j (-x) = -\lip_j (x)$ holds for
	$j \in [n], x \in S_{\sigma(j)}$. Since $\pi_{\sigma(i)} z = -\pi_{\sigma(i)} y$, we have that
  $\lip_i (\pi_{\sigma(i)} z) = -\lip_i (\pi_{\sigma(i)} y)$, from which
  $\pi_i \lip(z) = -\pi_i \lip(y) \in S_i$ by the homogeneity of $\lip$ in $k$ components.

  We have that
  \[ \norm{\pi_i \lip(y) - \pi_i x} \leq \norm{\lip(y) - x} \leq \norm{y - \ihom(x)} = 1 - \norm{\pi_{\sigma(i)} \ihom(x)}. \]
	Similarly, we have that
  \[ \norm{-\pi_i \lip(y) - \pi_i x} = \norm{\pi_i \lip(z) - \pi_i x} \leq \norm{\lip(z) - x} \leq \norm{z - \ihom(x)} = 1 + \norm{\pi_{\sigma(i)} \ihom(x)}. \]

  This means that $\pi_i x$ lies in the
	intersection of the two balls $B(\pi_i \lip(y), 1-\norm{\pi_{\sigma(i)} \ihom(x)})$ and
	$B(-\pi_i \lip(y), 1+\norm{\pi_{\sigma(i)} \ihom(x)})$. The intersection of these is a convex
	set in $X_i$, and is contained in the sphere of each ball, since the distance between their centers is the sum of their radii. Since $X_i$ is a strictly
	convex space, this intersection can contain at most one point. Since
  $\norm{\pi_{\sigma(i)} \ihom(x)}\pi_i \lip(y)$ belongs to both balls, we must have that
	$\norm{\pi_{\sigma(i)} \ihom(x)}\pi_i \lip(y) = \pi_i x$. Since we had $\norm{\pi_i \lip(y)}=1$
	by the homogeneity of $\lip$, this gives us
  $\norm{\pi_i x} = \norm{\pi_{\sigma(i)} \ihom(x)}$.

	If $\pi_i x = 0$, we are done, since this gives us
	$\pi_{\sigma(i)} \ihom(x) = 0$. We shall proceed with the assumption that
  $\pi_i x \neq 0$.

	In this case, we have that
  $\norm{\pi_i x}\lip_i (\pi_{\sigma(i)} y) = \pi_i x$.
	Dividing both sides by $\norm{\pi_i x}$ and applying $g_i$, we get
	$\pi_{\sigma(i)} y = \ihom_i (\frac{\pi_i x}{\norm{\pi_i x}})$.
	By the definition of $y$, this gives us
  \[ \frac{\pi_{\sigma(i)} \ihom(x)}{\norm{{\pi_{\sigma(i)} \ihom(x)}}} = \frac{\pi_{\sigma(i)} \ihom(x)}{\norm{\pi_i x}} = \ihom_i \frac{\pi_i x}{\norm{\pi_i x}}, \]
	from which
	${\pi_{\sigma(i)} \ihom(x)} = \norm{\pi_i x}\ihom_i (\frac{\pi_i x}{\norm{\pi_i x}}) = \pi_{\sigma(i)} \gamma(x)$
	is immediate.
\end{proof}

We now have enough to prove \autoref{thm:natural}.

\begin{proof} 

  By \autoref{lem:bijective-factors}, we have that the functions $\lip_i$ exist,
  so $\phi$ is well-defined.
  It's clear by the construction of $\phi$ that $\phi$ is a bijection, and $\phi(\alpha x) = \alpha \phi(x)$
  for all $x \in S_Z, \alpha \in [-1,1]$. Moreover, because $\gamma(S_Z) \subseteq S_Z$ and $\gamma(B_Z\setminus S_Z) \subseteq B_Z \setminus S_Z$, we have $\gamma(S_Z) = S_Z$, so $\phi(S_Z) = S_Z$.

  By \autoref{lem:sphere-implies-extreme}, we have $\lip(E) \subseteq E$.
  By \autoref{lem:homogeneity-equiv} and \autoref{lem:g-homogeneous}, we have $\lip = \phi$,
  so $\lip$ satisfies the same properties stated above for $\phi$.

By \autocite[Lemma 2.5]{cascales:2016}
  the facts just stated about $\lip$, along with the fact that $\lip$ is $1$‑Lipschitz, are sufficient for $\lip$ to be an isometry of $B_Z$.
\end{proof}

\section*{Acknowledgements}

This work was supported by the Estonian Research Council grants PRG1901. The author thanks Rainis Haller and Nikita Leo for their comments and support.


\printbibliography
\end{document}